\documentclass{article}
\usepackage{amssymb,amsfonts,latexsym,wasysym}

\newtheorem{theorem}{Theorem}[section]
\newtheorem{remark}[theorem]{Remark}
\newtheorem{lemma}[theorem]{Lemma}

\newtheorem{prop}[theorem]{Proposition}
\newtheorem{cor}[theorem]{Corollary}
\newtheorem{ex}[theorem]{Examples}
\newtheorem{ex1}[theorem]{Example}
\newtheorem{defi}[theorem]{Definition}

\renewcommand{\c}{\mathbb{M}}
\newcommand{\Aut}{\mathrm{Aut}}

\newcommand{\tp}{\mathrm{tp}}

\newcommand{\acl}{\mathrm{acl}}
\newcommand{\dcl}{\mathrm{dcl}}
\newcommand{\dom}{\mathrm{dom}}
\newcommand{\rng}{\mathrm{rng}}
\newcommand{\SU}{\mathrm{SU}}
\newcommand{\acli}{\mathrm{acl}^{\mathrm{eq}}}
\newcommand{\Th}{\mathrm{Th}}
\newcommand{\g}[1]{\mathfrak{#1}}
\newcommand{\On}{\mathrm{On}}
\newcommand{\DD}{\mathbb{D}}
\newcommand{\HH}{\mathbb{H}}
\newcommand{\PP}{\mathbb{P}}

\def\Ind#1#2{#1\setbox0=\hbox{$#1x$}\kern\wd0\hbox to 0pt{\hss$#1\mid$\hss}
\lower.9\ht0\hbox to 0pt{\hss$#1\smile$\hss}\kern\wd0}
\def\ind{\mathop{\mathpalette\Ind{}}}
\def\Notind#1#2{#1\setbox0=\hbox{$#1x$}\kern\wd0\hbox to 0pt{\mathchardef
\nn=12854\hss$#1\nn$\kern1.4\wd0\hss}\hbox to
0pt{\hss$#1\mid$\hss}\lower.9\ht0 \hbox to
0pt{\hss$#1\smile$\hss}\kern\wd0}

\newenvironment{proof}{\vspace{-0.25cm}
{\bf Proof}: }{\hfill $\Box$}

\begin{document}
\title{Orbits of subsets of the monster model and geometric theories}
\author{Enrique Casanovas and Luis Jaime Corredor\thanks{The first  author was partially supported by the Spanish government grant MTM2014-59178-P and by the Catalan government grant  2014SGR-437. The second author was partially supported by Fundaci\'on para la Promoci\'on de la Investigaci\'on y la Tecnolog\'{\i}a from Banco de la Rep\'ublica,  Colombia, grant N.~3.610 and by Facultad de Ciencias from Universidad de los Andes, Bogot\'a, proyecto semilla 2014-2.}}
\date{January 31, 2017. Revised June 1, 2017}
\maketitle

\begin{abstract}   Let $\c$  be the monster model of a complete first-order theory $T$. If  $\DD$  is a subset of $\c$,  following D. Zambella we consider $e(\DD)=\{\DD^\prime\mid (\c,\DD)\equiv (\c,\DD^\prime)\}$ and  $o(\DD)=\{\DD^\prime\mid (\c,\DD)\cong (\c,\DD^\prime)\}$. The general question we ask is when $e(\DD)=o(\DD)$ ?  The case where $\DD$ is $A$-invariant for some small set $A$  is rather straightforward: it just means that $\DD$ is definable. We investigate the case where $\DD$ is not invariant over any small subset. If T is geometric and $(\c,\DD)$ is an $H$-structure (in the sense of A. Berenstein and E.~ Vassiliev) we get some answers. In the case of $SU$-rank one, $e(\DD)$ is always different from $o(\DD)$. In the o-minimal case, everything can happen, depending on the complexity of the definable closure.  We also study the case of lovely pairs of geometric theories.

\end{abstract}

\section{Subsets of the monster model}

We will consider complete first-order theories $T$ having an infinite model.  By a monster model of $T$ we understand a model  $\c$ of $T$ whose universe is a proper class  and realizes all types over all subsets.   Every theory has a monster model and it is unique up to isomorphism.  Alternatively, one can assume the existence of a strongly inaccesible cardinal $\lambda$ and take a saturated model of cardinality $\lambda$ as the monster model of $T$. In this last case only set of cardinality less than $\lambda$ will be considered.    A set is \emph{small} if its cardinality is smaller than $|\c|$, which means  that it is just a set (not a proper class) if we understand $\c$ as a proper class model and that it has cardinality $<\lambda$ if we consider $\c$ as a saturated model of cardinality $\lambda$.  The requirements on $\lambda$ for the existence of a saturated model of cardinality $\lambda$ are weaker than strong inaccesibility (see Theorem VIII.4.7 in~\cite{Sh90a}), but for some results  we will need to use the assumption  that if $A$ is small, its power set is small too.

The language of $T$ is $L$. If $\varphi$ is a formula, $\varphi\in L$ means that the symbols of $\varphi$  belong to $L$.  We consider $n$-ary relations $\DD$ on $\c$ and  corresponding expansions $(\c,\DD)$  to the language $L\cup\{D\}$, where $D$ is a new $n$-ary predicate symbol.     There are two classes of relations naturally associated to $\DD$, the class $o(\DD)$ of all images of $\DD$ under automorphisms of $\c$ and the class $e(\DD)$ of all relations with expansion  elementarily equivalent to $(\c,\DD)$.  If the relation $\DD$ consists of a single $n$-tuple  $\overline{a}$, then $e(\DD)$ is the class of all $n$-tuples having the same type as $\overline{a}$ over the empty set and clearly $o(\DD)=e(\DD)$. The main question we address here was raised by D.~Zambella and it is to explain the meaning of $o(\DD)=e(\DD)$. There are two cases to be considered. If $\DD$ is invariant over some small set, then the answer is more or less straightforward: it just means that $\DD$ is definable. The case where $\DD$ is not invariant over any small set  does not seem to have been investigated. We will explore the question in the setting of geometric theories, considering two particular expansions: $H$-structures and lovely pairs. We will offer some partial answers illustrating different behaviour in the case of $\SU$ rank one theories and in the case of o-minimal theories.

\begin{defi}\rm  Let $\DD$ be an $n$-ary relation on the monster model $\c$.    For any small subset $A$ of $\c$ we define:
\begin{itemize} 
\item $o(\DD/A)=\{ \DD^\prime \subseteq \c^n\mid (\c,\DD)\cong_A (\c,\DD^\prime)\}= \{f(\DD)\mid f\in\Aut(\c/A)\}$ 
\item $e(\DD/A)=\{\DD^\prime\subseteq \c^n\mid (\c,\DD)\equiv_A (\c,\DD^\prime)\}$. 
\end{itemize}
We call $o(\DD/A)$ the \emph{orbit} of $\DD$ over $A$ and we call $e(\DD/A)$  the \emph{elementary class} of $\DD$ over $A$.
\end{defi}

\begin{remark} $o(\DD/A)$ and $ e(\DD/A)$  are classes of the equivalence relations $\cong_A$ and $\equiv_A$ on $\mathcal{P}(\c^n)$ and each class  $e(\DD/A)$  splits into  $\cong_A$-classes, one of them being  $o(\DD/A)$.
Note that $| \{e(\DD/A)\mid \DD\subseteq \c^n\}|\leq 2^{|T(A)|}$
\end{remark}

Working in $T(A)$  we may sometimes assume $A=\emptyset$.  We will use the notation  $o(\DD)= o(\DD/\emptyset)$  and $e(\DD)=e(\DD/\emptyset)$. 

 These notions have been discussed by D.~Zambella in~\cite{Zambella14} and in a talk given at the Barcelona Logic Seminar. Some of the results presented in this section were known to him, in particular (with minor differences and with different proofs) a great part of  Remark~\ref{s1},  propositions~\ref{s3} and~\ref{s6} and Corollary~\ref{s4}.  Very likely theseresults are generally known and should be considered folklore. We thank R. Farr\'e  for some useful discussions concerning this section.

\begin{defi}\rm We say that $\DD$ is \emph{saturated} if  the expansion $(\c,\DD)$  is saturated, that is, if this expansion  is the monster model of $\Th(\c,\DD)$.
\end{defi}

\begin{remark}\label{s1} 
\begin{enumerate}
\item  For any $\DD$ there is always some saturated $\DD^\prime\in e(\DD)$. 
\item If $\DD$ is saturated, then $o(\DD)=\{\DD^\prime \in e(\DD) \mid \DD^\prime \mbox{ is saturated  }\}$. 
\item $e(\DD)=o(\DD)$ if and only if every $\DD^\prime\in e(\DD)$ is saturated.
\end{enumerate}
\end{remark}
\begin{proof}  If $T^\prime$ is a complete and consistent extension of $T$ in a language $L^\prime\supseteq L$, then the monster model $\c$ of $T$ can be expanded  to a monster model  $\c^\prime$  of $T^\prime$. Moreover, if $(\c,\DD)\equiv(\c,\DD^\prime)$ are monster models, then $(\c,\DD)\cong (\c,\DD^\prime)$.
\end{proof}

\begin{remark}\label{s2}
\begin{enumerate}
\item If $e(\DD)=o(\DD)$, then  $e(\DD/A)=o(\DD/A)$ for every small set $A$.
\item If $e(\DD/A)=o(\DD/A)$ for some small set $A$, then $\DD$ is saturated.
\end{enumerate}
\end{remark}
\begin{proof} Notice that for any small set $A$, $(\c,\DD)$ is saturated iff $(\c_A,\DD)$ is saturated. Hence \emph{1} follows  from  item 3 of Remark~\ref{s1}. For \emph{ 2}, observe that the assumption together with item 1 of Remark~\ref{s1} imply that $(M_A,\DD)$ is saturated.  
\end{proof}

We don't know if item \emph{2} of Remark~\ref{s2} can  be strengthened to:  if $e(\DD/A)=o(\DD/A)$ for some small set $A$, then $e(\DD)=o(\DD)$.  We conjecture that it is not true in general, but it is true if $\DD$ is $B$-invariant for some small set $B$, as follows easily from Proposition~\ref{s3} below.

\begin{defi} \rm  An $n$-ary relation $\DD$ in $\c$ is \emph{definable} if it is definable over some small set $A$, in which case we also say that it is \emph{$A$-definable}. In the case $A=\emptyset$ we say  it is  	\emph{$0$-definable}.  The relation  $\DD$ is \emph{$A$-invariant} if every $A$-automorphism of the monster model $\c$   fixes $\DD$ setwise, that is, if $o(\DD/A)=\{\DD\}$.
\end{defi}

 The next proposition explains the meaning of  $e(\DD)=o(\DD)$   when $\DD$ is $A$-invariant for some small set $A$: it means that $\DD$ is $A$-definable and also means that $\DD$ is saturated.  The same question is much more complicated if $\DD$ is not $A$-invariant for any small set $A$, as will be seen in the next sections. As a side remark note that if $\DD$ is $A$-invariant for some small set $A$, then $\DD$ is not small, unless $\DD\subseteq\acl(A)$.

\begin{prop}\label{s3} The following are equivalent for any small set $A$:
\begin{enumerate}
\item $\DD$ is $A$-definable.
\item $e(\DD)=o(\DD)$ and $\DD$ is $A$-invariant.
\item $e(\DD/A)= \{\DD\}$
\item $\DD$ is saturated and $\DD$ is $A$-invariant.
\end{enumerate}
\end{prop}
\begin{proof} \emph{1} $\Rightarrow$ \emph{2}. If $\DD$ is $A$-definable, it is clearly $A$-invariant.   Moreover, it is saturated: replace the symbol $D$ by the formula defining $\DD$ to check it.  Since every $\DD^\prime\in e(\DD)$ is also definable, it  follows that every  $\DD^\prime\in e(\DD)$ is saturated. By item 3 of Remark~\ref{s1},  $e(\DD)=o(\DD)$.

\emph{2} $\Rightarrow$ \emph{3}. By  Remark~\ref{s2}, $e(\DD/A)=o(\DD/A)$, and by $A$-invariance $o(\DD/A)=\{\DD\}$.

\emph{3} $\Rightarrow$ \emph{4}. The assumption implies $e(\DD/A)=o(\DD/A)=\{\DD\}$ and, by  Remark~\ref{s1},  $\DD$ is saturated.

\emph{4} $\Rightarrow$ \emph{1}.  Assume $\DD$ is $A$-invariant and saturated. By $A$-invariance, $\DD$ is a union $\bigcup_{i\in I} \pi_i(\c)$ where every $\pi_i(x)$ is an  $L$-type over $A$ (take $\pi_i(x)=\tp(a_i/A)$, where $\DD=\{a_i\mid i\in I\}$). Working in the monster model $(\c,\DD)$ we see that for some finite $I_0\subseteq I$, for some finite $\pi_{i0}\subseteq \pi_i$,  $(\c,\DD)\models \forall x(D(x)\leftrightarrow \bigvee_{i\in I_0}\bigwedge \pi_{i0}(x))$
\end{proof}

 \begin{cor} \label{s4} 
$\DD$  is $0$-definable iff   $e(\DD)=\{\DD\}$. 
\end{cor}
\begin{proof}  By Proposition~\ref{s3} with $A=\emptyset$.
\end{proof}

\begin{cor}\label{s9} If  $A$ is a small set and $\DD\subseteq \c^n$ is not $A$-definable, then for some (all) saturated $\DD^\prime\in e(\DD/A)$, $o(\DD^\prime/A)\neq \{\DD^\prime\}$.
\end{cor}
\begin{proof}  Assume $\DD^\prime\in e(\DD/A)$ is saturated and $o(\DD^\prime/A)=\{\DD^\prime\}$, that is,  $\DD^\prime$ is $A$-invariant.  By Proposition~\ref{s3}, $\DD^\prime$ is $A$-definable. Then $\DD$ is also $A$-definable.
\end{proof}

\begin{ex}\label{s5}
\begin{enumerate}
\item  Consider the theory of an equivalence relation $E$ with infinitely many classes, all infinite. Each equivalence class  $\DD$  is definable and therefore $e(\DD)=o(\DD)$. Note that $o(\DD)$ is the set of all equivalence classes.
\item  Let $(\c ,<)$ be the monster model of the theory of a dense linear order without end points and let  $\HH$ be a dense and co-dense subset. As shown in propositions~\ref{h5} and~\ref{o3},  $\HH$ is not invariant over any small set and $e(\HH)=o(\HH)$.
\item Let $\c$ be the monster model of the theory of algebraically closed fields of characteristic zero and let $\mathbb{Q}$ the set of rational numbers. Then $\mathbb{Q}$ is small, $0$-invariant, not definable, and $e(\mathbb{Q})\neq  o(\mathbb{Q})$. Its complement $\mathbb{I}=\c\smallsetminus \mathbb{Q}$ is type-definable over $\emptyset$ but again $e(\mathbb{I})\neq  o(\mathbb{I})$. 
\end{enumerate}
\end{ex}

\begin{prop}\label{s6} The following are equivalent for  any saturated $\DD$:
\begin{enumerate}
\item $\DD$ is definable.
\item  $o(\DD/A)=\{\DD\}$    for some  finite (small) set $A$.
\item $o(\DD/A)$  is finite  for some  finite (small)  set $A$.
\item $o(\DD/A)$ is small for some finite (small) set $A$.
\end{enumerate}
\end{prop}
\begin{proof}  Clearly \emph{1} implies \emph{2} (if $\DD$  is definable over the finite set $A$, then $o(\DD/A)=\{\DD\}$),  \emph{2} implies \emph{3},  and \emph{3} implies \emph{4}.

\emph{4} $\Rightarrow$  \emph{3}.  Assume $o(\DD/A)$  is infinite.   Add new  $n$-ary predicates $D_i$ and $1$-ary function symbols $F_i$  to the language for every $i<\kappa$. Let $\Sigma$ be a set of sentences saying that $D_i\neq D_j$  for every $i\neq j$ and  saying that $F_i$ is an  $A$-isomorphism of $\c$ mapping $D$ onto $D_i$. Then $\Sigma\cup\Th(\c,\DD)$ is consistent and  it is satisfied in some  expansion $\c^\prime$ of $(\c_A,\DD)$.  The relations $\DD_i=D_i^{\c^\prime}$  are different $A$-conjugates of $\DD$ for every $i<\kappa$.

\emph{3} $\Rightarrow$ \emph{1}. Assume $\DD$ is not definable.  In order to contradict  \emph{3}, we will prove prove that $o(\DD/A)$ is infinite for every small set $A$. We prove by induction on $n$ that $|o(\DD/A)|\geq n$.  Inductively assume that $\DD_1,\ldots,\DD_n$ are different $A$-conjugates of $\DD$, with $\DD=\DD_1$. Choose tuples $a_i$ such that $a_i\in\DD\Delta \DD_i$   and let   $A^\prime =A\cup\{a_i\mid i=2,\ldots, n\}$.  Since $\DD$ is not $A^\prime$ definable and it is saturated,  we may apply   Corollary~\ref{s9}   to get  some $\DD^\prime\in o(\DD/A^\prime)\smallsetminus \{\DD\}$.  Since $a_i\in A^\prime$, $a_i$ witnesses that $\DD^\prime\neq \DD_i$. Since $A\subseteq A^\prime$, $\DD^\prime\in o(\DD/A)$.  Hence $|o(\DD/A)|\geq n+1$.
\end{proof}

\begin{remark}\label{as5} $e(\DD)$ is finite if and only if it is small.
\end{remark} 
\begin{proof} Assume $e(\DD)$ is infinite. If $\Sigma(D)=\Th(\c,\DD)$, then for every $\kappa$ the following is consistent $\bigcup_{i<\kappa}\Sigma(D_i)\cup \{\neg \forall x(D_i(x)\leftrightarrow D_j(x))\mid i<j<\kappa\}$.
\end{proof}

\begin{cor}\label{s7} The following are equivalent for  any  $\DD$:
\begin{enumerate}
\item $\DD$ is definable.
\item $e(\DD/A)=\{\DD\}$ for some finite (small) set $A$.
\item $e(\DD/A)$  is finite  for some  finite (small) set $A$.
\item $e(\DD/A)$ is small for some finite (small) set $A$.
\item $o(\DD)=e(\DD)$ and $o(\DD/A)$ is small for some finite  (small) set  $A$.
\item $\DD$ is saturated and $o(\DD/A)$ is small for some finite (small) set $A$.
\end{enumerate}
\end{cor}
\begin{proof} \emph{1} $\Rightarrow$ \emph{2}. By Proposition~\ref{s3}.

\emph{2} $\Rightarrow$ \emph{3}  and  \emph{3} $\Rightarrow$ \emph{4}  are immediate.

\emph{4} $\Rightarrow$ \emph{1}.  Choose some saturated $\DD^\prime\in e(\DD)$ and notice that $e(\DD^\prime/A)$ and $o(\DD^\prime/A)$ are small. By Proposition~\ref{s6}, $\DD^\prime$ is definable.  Hence $\DD$ is definable too.

\emph{1} $\Rightarrow$ \emph{5}.  By Proposition~\ref{s3}.

\emph{5} $\Rightarrow$ \emph{6}. By Remark~\ref{s2}.

\emph{6} $\Rightarrow$ \emph{1}.  By Proposition~\ref{s6}.
\end{proof}

\section{$H$-structures}

In this section $T$ is a  \emph{geometric theory}, which means that the algebraic closure operator $\acl$ defines a pregeometry (has the exchange property) and $T$ eliminates the quantifier $\exists^\infty$ (``there exist infinitely many'').  For example, if $T$ is o-minimal or has $\SU$-rank one, then it is geometric.  In particular, strongly minimal theories are geometric. Geometric theories were introduced by E.~Hrushovski and A.~Pillay in~\cite{HrushovskiPillay94} and they were further investigated by J. Gagelman in~\cite{Gagelman05}.  In any geometric theory  we have a well defined algebraic dimension and we have an independence relation between subsets of the monster model defined by:  $A\ind_B C$  iff  $\dim(A_0/B)=\dim(A_0/BC)$  for every finite $A_0\subseteq A$. An equivalent definition is:  $A\ind_B C$ iff  every subset of $A$  which is algebraically independent over $B$  is also algebraically independent over $BC$.

$H$-structures were first considered by   A.~Dolich, C.~Miller and C.~Steinhorn in the context of some  o-minimal theories in~\cite{DolichMillerSteinhorn16}    and  were  fully investigated in the setting of geometric theories by A.~Berenstein and E.~Vassiliev in~\cite{BerensteinVassiliev12}. Propositions~\ref{h1} and~\ref{h2} and Corollary~\ref{h3} are versions of results of  this last article. Since there are some modifications and moreover we want to make our presentation self-contained, we will give short proofs of these results.

We add a new unary predicate $H$ and consider structures   $(M,H^M)$  in the language $L_H= L\cup\{H\}$  where  $M\models T$.   For any subset $A\subseteq M$,  we use the notation $H(A)= A\cap H^M$. In particular, $H(M)=H^M$.  Syntactically, we use sometimes the short notation $(\exists x\in H)\varphi$  and  $(\forall x\in H)\varphi$ for  $\exists x(H(x)\wedge \varphi)$  and  $\forall x(H(x)\rightarrow \varphi)$ respectively.

\begin{defi} \rm  A structure  $(M,H(M))$  is an \emph{$H$-structure}  if  $H(M)$  is a subset of algebraically independent elements and the following two conditions are satisfied:
\begin{enumerate}
\item Density: If $A\subseteq M$ is finite  and $q(x)\in S_1(A)$ is a non-algebraic $L$-type, then it has some realization  $a\models q$ in $H(M)$.
\item Extension:  If $A\subseteq M$ is finite  and $q(x)\in S_1(A)$ is a non-algebraic $L$-type, then it has some realization  $a\models q$ in $M$   such that $a\not\in \acl(AH(M))$.
\end{enumerate}
\end{defi}

\begin{prop}\label{h1}  Let $(M,H(M))$ and $(N,H(N))$  be $H$-structures and let $I$ be the set of all   $L$-elementary mappings $f$ from $M$ into $N$  with finite domain $A=\dom(f)$   in $M$ and range $B=\rng(f)$ in $N$ and such that  
\begin{enumerate}
\item $a\in H(M)$  iff  $f(a)\in H(N)$  for all $a\in A$
\item $A\ind_{H(A)} H(M)$  and $B\ind_{H(B)}H(N)$.
\end{enumerate}
Then $(M,H(M))$ and $(N,H(N))$  are partially isomorphic via $I$.
\end{prop}
\begin{proof} Let $a\in M$, $f\in I$ and let us show that $f$ can be extended to  some $f^\prime\in I$  with $a\in\dom(f^\prime)$.  Adding some element to the range is similar. There are several cases to be considered:

\emph{Case 1}. $a\in\acl(A)$. Let $p(x)=\tp(a/A)$ and  let  $b\in N$ realize the conjugate  type  $p^ f(x)=\{\varphi(x,f(\overline{c})\mid \varphi(x,\overline{c})\in p\}$.  Let  $A^\prime =Aa$ and $B^\prime =Bb$  and let $f^\prime =f\cup\{(a,b)\}$.  We claim that  $f^\prime \in I$. Since $A\ind_{H(A)} H(M)$, if  $a\in H(M)$, then  $a\in A$  (because $a\ind_{H(A)}a$ implies $a\in\acl(H(A))$, and since $H(M)$ is algebraically independent, this is  only possible if $a\in H(A)$).  Similarly, if $b\in H(N)$, then  $b\in B$.  Hence, $a\in H(M)$  iff  $b\in H(M)$.  If   $a\in H(M)$, then $a\in A$ and   we have  $H(A^\prime) = H(A)a= H(A)$. If   $a\not\in H(M)$,   then  clearly $H(A^\prime)=H(A)$. In both cases   $Aa\ind_{H(A)} H(M)$.  Similarly for $b$. 

\emph{Case 2}. $a\in H(M)$. We can assume $a\not\in \acl(A)$. Let $p(x)=\tp(a/A)$.  By the density property, there is some realization $b$ in $H(N)$ of  $p^f(x)$.  Let $A^\prime =Aa$, $B^\prime =Bb$  and $f^\prime  =f\cup\{(a,b)\}$. Clearly,   $a\in H(M)$  iff   $b\in H(N)$.  Note that  $H(A^\prime)= H(A) a$.   Since $a\in H(M)$,   $A\ind_{H(A)a} H(M)$. It follows that $Aa\ind_{H(A)a}H(M)$.  Similarly for $b$.

\emph{Case 3}.  $a\in \acl(AH(M))$.  There is a finite tuple $\overline{h}\in H(M)$  such that  $a\in\acl(A\overline{h})$.  By  case \emph{2}, there is some extension $f^\prime\in I$ of $f$  with  the tuple $\overline{h}$ in its domain $A^\prime$.  Then  $a\in \acl(A^\prime)$ and we can apply case \emph{1}.

\emph{Case 4}.  $a\not\in\acl(AH(M))$.  Let $p(x)=\tp(a/A)$. By the extension property,   there is a  realization $b\in N$  of  $p^f$  such that $b\not\in\acl(B H(N))$. Let be $A^\prime= Aa$,  $B^\prime=Bb$ and  $f^\prime=f\cup\{(a,b)\}$.  Clearly,  $a\in H(M)$  iff   $b\in H(N)$.  Since  $a\not\in H(M)$,   $H(A^\prime)= H(A)$. Since $a\ind A H(M)$,  it follows that $Aa\ind_{H(A)} H(M)$. Similarly for $b$.
\end{proof}

\begin{defi}\rm \label{axioms} Let $T^\mathrm{indep}$ be the theory with the following set of axioms:
\begin{enumerate}
\item  $H$  is a set of independent elements:   $(\forall x_1\in H)\ldots (\forall x_n\in H)(\forall x\in H)(\bigwedge_{i=1}^n x\neq x_i \wedge \exists^{<k}x\psi(x_1,\ldots,x_n,x)\rightarrow \neg\psi(x_1,\ldots,x_n,x))$.\\
for every $\varphi(x_1,\ldots,x_n,x)\in L$, $\psi(x_1,\ldots,x_n,x)\in L$, for every  $n,k$.
\item Density: $\forall x_1\ldots  x_n ( \exists^\infty x  \varphi(x_1,\ldots,x_n,x)  \rightarrow (\exists x\in H)\varphi(x_1,\ldots,x_n,x))$\\ for every $\varphi(x_1,\ldots,x_n,x)\in L$, for every $n$.
\item Extension: $\forall x_1\ldots x_n(\forall  y_1\ldots y_m\exists^{<k}x\,\psi(x_1,\ldots,x_n,y_1,\ldots,y_m,x)  \wedge  \exists^\infty x  \varphi(x_1,\ldots,x_n,x) $\\ $ \rightarrow \exists x( \varphi(x_1,\ldots,x_n,x)\wedge (\forall y_1\in H)\ldots (\forall y_m\in H)\neg \psi(x_1,\ldots,x_n,y_1,\ldots,y_m,x))) $
\\ for every $\varphi(x_1,\ldots,x_n,x)\in L$,  $\psi(x_1,\ldots,x_n,y_1,\ldots,y_m,x)\in L$, and  $n,m,k$.
\end{enumerate}
\end{defi}

\begin{prop}\label{h2}  All  $H$-structures are models of $T^\mathrm{indep}$ and  any $\omega$-saturated model of  $T^\mathrm{indep}$  is a $H$-structure.
\end{prop}
\begin{proof}  Clear.
\end{proof}

\begin{cor}\label{h3} All $H$-structures are elementarily equivalent and $T^\mathrm{indep}$ is its common complete theory.
\end{cor}
\begin{proof} On the one hand, all $H$-structures are back-and-forth equivalent and satisfy $T^\mathrm{indep}$. On the other hand, every model of $T^\mathrm{indep}$ has an $\omega$-saturated elementary extension and by Proposition~\ref{h2} this extension is an $H$-structure.
\end{proof}

 We will call $T^\mathrm{indep}$ the theory of H-structures of $T$. If   $(\c, H(\c))$  is the monster model of  $T^\mathrm{indep}$, in order to simplify notation we write  $\HH=H(\c)$.

\begin{lemma} \label{h4}  Let $T$ be a geometric theory,  $\c$ the monster model of $T$ and  $(\c,\HH)$  the monster model of the corresponding theory of $H$-structures. 
\begin{enumerate}
\item Density: for every small set $B\subseteq \c$,   every  non-algebraic $L$-type $p(x)\in S_1(B)$  has a realization in $\HH$.
\item Extension: for every small set $B\subseteq \c$,   every  non-algebraic $L$-type $p(x)\in S_1(B)$ has a realization $b$ in $\c$  such that  $b\not\in \acl(B \HH)$.
\item If $(\c^\prime,\HH^\prime)\equiv (\c,\HH)$ has the  properties 1 and 2, then  $(\c,\HH)\cong (\c^\prime,\HH^\prime)$.
\end{enumerate}
\end{lemma}
\begin{proof} Being an $\omega$-saturated model of its own theory, $(\c,\HH)$ is an $H$-structure.  This implies easily \emph{1}.

\emph{2}.  Let $B\subseteq \c$ and let $p(x)\in S_1(B)$  be a non-algebraic $L$-type.  For every finite $B_0\subseteq B$  there is some $b\models p\restriction B_0$ in $\c$  such that $b\not\in \acl(B_0 \HH)$.  There is an $L_H$-type $\pi(x)$ over $B$ such that  $b$  realizes $\pi(x)$  iff  $b\not\in \acl(B\HH)$.  Hence, some realization of $p$  in $\c$  is non-algebraic over $B\HH$.

\emph{3}. Let $I$ be the set of all $L$-elementary mappings $f$   from $\c$ into $\c^\prime$ preserving $H$ and whose domain $A$ and range $B$ are small sets,  $A\ind_{H(A)}\HH$  and  $B\ind_{H(B)}\HH^\prime$.  It is easy to see that $I$ is a back-and-forth system between $(\c,\HH)$  and  $(\c^\prime,\HH^\prime)$.  Therefore we can construct an ascending chain $(f_\alpha\mid \alpha\in \On)$ of mappings $f_\alpha\in I$ such that every element of $\c$  is in the domain of some $f_\alpha$ and every element of $\c^\prime$ is in the range of some $f_\alpha$.  Then $f=\bigcup_{\alpha\in \On} f_\alpha$  is an isomorphism from $(\c,\HH)$  onto  $(\c^\prime,\HH^\prime)$.
\end{proof}

The next proposition shows that $H$-structures of geometric theories provide a good framework to study the problem $o(\HH)=e(\HH)$, since $\HH$ is not $A$-invariant over any small set $A$.

\begin{prop}\label{h5}  Let $T$ be a geometric theory,  $\c$ the monster model of $T$ and  $(\c,\HH)$  the monster model of the theory of $H$-structures.   For any small set $A$,  $o(\HH/A)\neq \{\HH\}$.
\end{prop} 
\begin{proof} Choose a non-algebraic type $p(x)\in S_1(A)$. By the  Lemma~\ref{h4}  there are realizations $a,b$ of $p$ in $\c$ such that $a\in \HH$ and  $b\not\in\acl(A\HH)$. Since $a\equiv_A b$, there is an automorphism $f\in\Aut(\c/A)$ with $f(a)=b$.  Then $f(\HH)\in o(\HH/A)\smallsetminus \{\HH\}$.
\end{proof}

\section{SU-rank one case}

$SU$-rank one theories are supersimple and are geometric. As supersimple theories, they eliminate hyperimaginaries and hence the Independence Theorem holds for arbitrary strong types.
This will be used in the next result. The proof splits into two cases. The first one corresponds to the example of strongly minimal theories, such as  vector spaces. The second one is inspired by the example of the random graph.

\begin{theorem}\label{su1} Let $T$  have $SU$-rank one and let $\c$ be the monster model of $T$. If $(\c,\HH)$  is the monster model of the corresponding theory of $H$-structures, then  $e(\HH)\neq o(\HH)$.
\end{theorem}
\begin{proof}   \emph{Case 1}.   For every  non-algebraic $p(x)\in S_1(\acli(\emptyset))$  there is a unique global non-algebraic type  $\g{p}\in S_1(\c)$  extending $p$  (formally,  $p$ is a subset of the unique extension of $\g{p}$ to a global type in $T^\mathrm{eq}$).  Let  $(p_i\mid i<\kappa)$ enumerate the non-algebraic elements of $S_1(\acli(\emptyset))$ and let us choose inductively a sequence $(a_{ij}\mid j< \omega)$ of  realizations $a_{ij}$  of $p_i$  such that  $a_{ij}\not\in\acl(\{a_{kl}\mid k<i,  l<\omega)\}\cup\{a_{il}\mid l<j\})$.  Let $\HH^\prime =\{a_{ij}\mid i<\kappa, j<\omega\}$.  We claim that $(\c,\HH^\prime)$  is an $H$-structure.  This will imply  that $(\c,\HH)\equiv (\c,\HH^\prime)$.  Since $\HH^\prime$ is a small set, clearly  $(\c,\HH)\not\cong (\c,\HH^\prime)$. Hence, we will obtain $\HH^\prime \in e(\HH)\smallsetminus o(\HH)$, as desired.  Clearly, $\HH^\prime$ is a set of independent elements. Let us consider a non-algebraic $L$-type $p(x)$  over a finite set $B\subseteq \c$.  Let  $p^\prime\in S_1(B\HH^\prime)$  be a non-algebraic $L$-type  extending $p$.  Since $\HH^\prime$ is a small set, we can realize $p^\prime$ in $\c$.  If  $a\models p^\prime$, then $a\models p$ and $a\not\in\acl(B\HH^\prime)$.  Now we check that $p$ can also be realized by some element of $\HH^\prime$.  Extend  $p$  to some non-algebraic type  $p^\prime(x)\in S_1(\acli(B))$  and let $i<\kappa$  be such that  $p_i=p^\prime\restriction \acli(\emptyset)$.  We claim that for some $j<\omega$,  $a_{ij}\not\in\acl( B)$ and therefore realizes  $p$. Otherwise, $\{a_{ij}\mid j<\omega\}\subseteq \acl(B)$, which is a contradiction since  $\acl(B)$ has finite dimension and $\{a_{ij}\mid j<\omega\}$ has infinite dimension.

\emph{Case 2}.  There is a non-algebraic $p_\emptyset(x)\in S_1(\acli(\emptyset))$  having two different global non-algebraic extensions.   We  claim that there is an independent indiscernible sequence $(\overline{a}_i\mid i<\omega)$  where each $\overline{a}_i$ is an independent $n$-tuple  $\overline{a}_i=(a_{i0},\ldots,a_{i n-1})$  and   there are two different non-algebraic types   $p_i,q_i\in S_1(\acli(\overline{a}_i))$  extending $p_\emptyset$. In order to start the construction, we fix first  a finite set $C$ such that $p_\emptyset$ has two different non-algebraic extensions over $\acl(C)$ and then we choose a maximal sequence $c_0,\ldots,c_{n-1}$ of algebraically independent elements of $C$. The construction of the sequence is then straightforward, we start  with an independent $n$-tuple $\overline{a}_0=(a_{00},\ldots,a_{0 n-1})=(c_0,\ldots,c_{n-1})$  and two different non-algebraic types $p_0,q_0\in S_1(\acli(\overline{a}_0))$ extending $p_\emptyset$ and then we obtain an indiscernible independent sequence $(\overline{a}_i\mid i<\omega)$  starting with   $\overline{a}_0$.  Let  $A_0=\{a_{ij}\mid i<\omega,  j<n\}$  and let  $p^\ast(x)\in S_1(\acli(A_0))$ be a non-algebraic extension of $p_\emptyset$.  We plan to construct some $H$-structure $(\c^\prime,\HH^\prime)$  where $\c^\prime\preceq \c$ is a monster model of $T$,  $A_0\subseteq \HH^\prime$  and the type $p^\ast(x)$  is omitted in $\HH^\prime$. Since $\c\cong \c^\prime$, there is some $\HH^{\prime\prime}$  such that  $(\c,\HH^{\prime\prime})\cong (\c^\prime,\HH^\prime)$.   Then  $(\c,\HH^{\prime\prime})\equiv (\c,\HH)$  but  since the type corresponding to $p^\ast(x)$ in the isomorphism $\c^\prime\cong \c$ is omitted in $\HH^{\prime\prime}$, $(\c,\HH^{\prime\prime})\not \cong (\c,\HH)$.  We will obtain $\c^\prime$ as a union of a chain of small sets  $(A_\alpha\mid \alpha\in \On)$  and  $\HH^\prime$ as a union of corresponding subsets $H_\alpha \subseteq A_\alpha$. The chain will start with $H_0=A_0$ and it will satisfy the following conditions:
\begin{enumerate}
\item $A_{\alpha}\ind_{H_\alpha} H_{\alpha +1}$
\item  For every finite $B\subseteq A_\alpha$, every non-algebraic $L$-type  $p(x)\in S_1(B)$  has a realization in $H_{\alpha +1}$.
\item For every $B\subseteq A_\alpha$,  every $L$-type  $p(x)\in S_1(B)$  has a realization $a\in A_{\alpha +1}$  and if  $p$ is non-algebraic, then $a\not\in\acl(BH_{\alpha +1})$.
\item The elements of $H_\alpha$ are independent.
\item No  element of $H_\alpha$ realizes  $p^\ast$.
\end{enumerate}
Using condition 1 one can inductively show that  $A_\alpha\ind_{H_\alpha}H_\beta$  for all $\beta>\alpha$   and hence  that  $A_\alpha\ind_{H_\alpha} \HH^\prime$.  Using condition 3  one sees  that $\c^\prime$ is a monster model of $T$.  We want to show now that $(\c^\prime, \HH^\prime)$ is an $H$-structure. Conditions 2 and 4   are part of what is needed. Consider now a finite $B\subseteq A_\alpha$ and some non-algebraic $L$-type $p(x)\in S_1(B)$. By condition 3  we can get some realization $a\in A_{\alpha +1}$ of  $p$  such that  $a\not \in \acl(B H_{\alpha +1})$. Since  $A_{\alpha +1 }\ind_{H_{\alpha +1}}\HH^\prime$, it follows that $a\not\in \acl(B\HH^\prime)$, as required.

We finally show that the chains $(A_\alpha\mid \alpha\in\On)$  and $(H_\alpha\mid \alpha\in\On)$  exist.  Assume $A_\alpha$  and  $H_\alpha\subseteq A_\alpha$ have been already obtained.  We first construct $H_{\alpha +1}$  and then $A_{\alpha +1}$.  In the limit case we just take unions.  Let  $(r_i(x)\mid i<\kappa)$  enumerate all non-algebraic  $L$-types $r_i\in S_1(B_i)$ over finite subsets $B_i$  of $A_\alpha$.  We will find a realization $b_i$ of  $r_i$  such that  $b_i\not\models p^\ast(x)$ and $b_i\not\in \acl(A_\alpha\cup \{b_j\mid j<i\})$ and we will put  $H_{\alpha  +1}=  H_{\alpha}\cup \{b_i\mid i<\kappa\}$.  This will ensure that the elements of $H_{\alpha +1}$ are independent.  Let $b_i$  realize some non-algebraic extension of $r_i$  over  $\acl(B_i)$.  If  $b_i$ does not realize $p_\emptyset$ we can further assume that $b_i\not\in\acl(A_\alpha\cup\{b_k\mid k<i\})$ and we add it to $H_{\alpha +1}$. In this case it is clear that $b_i$ does not realize $p^\ast$.  Now assume that every realization   of $r_i$ that is not algebraic over $B_i$ realizes $p_\emptyset$.   Since  $\acl(B_i)$ has finite dimension, some tuple $\overline{a}_j$ is disjoint with $\acl(B_i)$.  It follows that $\overline{a}_j\ind B_i$.  There is a non-algebraic  type $q_j
\in S_1(\acli(\overline{a}_j))$  extending $p_\emptyset$ and different from $p^\ast\restriction \acli(\overline{a}_j)$. By the Independence Theorem, we can amalgamate  $q_j$  and  $r_i$  obtaining some common realization $b_i$ of these types such that  $b_i\not\in \acl(B_i \overline{a}_j)$.  We can clearly assume that additionally  $b_i\not\in \acl(A_\alpha\cup \{b_k\mid k< i\})$. We add  this element to $H_{\alpha +1}$.  Note that $b_i$ does not realize $p^\ast$. Note also that the construction satisfies $H_{\alpha +1}\ind_{H_\alpha} A_\alpha$.  Finally, we must extend  $A_\alpha \cup H_{\alpha +1}$ to  $A_{\alpha +1}$ in the following way:  we consider  all $L$-types $r(x)\in S(B)$ over arbitrary subsets $B$ of  $A_\alpha$  and for each such type we add  a realization $b$; moreover,  if the type $r$  is non-algebraic we additionally require that $b\not\in \acl(A_\alpha H_{\alpha +1})$. 
\end{proof}

\section{o-minimal case}

Now we consider o-minimal theories, another example of geometric theories.  By an o-minimal theory we understand here the theory of a densely ordered o-minimal structure. We will use constantly the fact that in any o-minimal theory,  $\dcl(A)=\acl(A)$ for any set $A$.  We will see that if $(\c,\HH)$ is the monster model of the theory of $H$-structures of models of an o-minimal theory, the equality $e(\HH) = o(\HH)$ holds in some cases and is false in some other cases, depending
on the complexity of the definable closure.

\begin{prop}\label{o0}Let $T$ be an o-minimal theory with a $0$-definable binary function $f$ such that
\begin{enumerate}
\item For every $x$ and $y$,  $f(x,y)$  is interdefinable with $y$ over $x$ and  with $x$ over $y$.
\item For every interval $(a,b)$  there is another interval $(a^\prime,b^\prime)$  such that if $x,y\in (a^\prime,b^\prime)$, then  $f(x,y)\in (a,b)$.
\end{enumerate}
Let $(\c,\HH)$ be the monster model of the theory of $H$-structures of $T$. If we extend $\HH$ to a basis $\HH^\prime$ of $\c$, then $(\c,\HH)\equiv (\c,\HH^\prime)$and  $(\c,\HH)\not\cong (\c,\HH^\prime)$ . Hence  $o(\HH)\neq e(\HH)$.
\end{prop}
\begin{proof} Note that $\HH$ is not a basis since any non-algebraic type over the empty set has a realization in $\c\smallsetminus \dcl(\HH)$. Hence $(\c,\HH)\not\cong (\c,\HH^\prime)$.
We will now check that $(\c,\HH^\prime)$ satisfies the axioms of the theory of $H$-structures as presented in Definition~\ref{axioms}.  This will imply $(\c,\HH)\equiv(\c,\HH^\prime)$. Clearly $\HH^\prime$ is a collection of independent elements and clearly every non-algebraic $L$-type $p(x)\in S_1(A)$ over a finite set $A\subseteq \c$ can be realized in $\HH^\prime$ (since it can be realized in  $\HH$ and $\HH\subseteq \HH^\prime$). Now we will check the extension  axioms.   Asume $a_1,\ldots,a_n\in \c$ and $\varphi(a_1,\ldots,a_n,\c)$ is infinite and therefore contains some interval $(b_1,b_2)$.   Let $\psi(a_1,\ldots,a_n,y_1,\ldots,y_m,x)$ be algebraic in $x$ for all $y_1,\ldots,y_m$.  We want to show that there is some $a$ in $(b_1,b_2)$ such that $\c\models \neg\psi(a_1,\ldots,a_n,h_1,\ldots,h_m,a)$  for all $h_1,\ldots,h_m\in\HH^\prime$.   Since $a_1,\ldots,a_n$ are definable over $\HH^\prime$ we may choose $h^\prime_1,\ldots,h^\prime_k\in \HH^\prime$  such that all $a_1,\ldots,a_n$ are definable over $h^\prime_1,\ldots,h^\prime_k$.  Now consider the $0$-definable function $f_{m+k+1}$ defined by iteration of $f$ as  $f_{m+k+1}(x_1,\ldots,x_{m+k+1})= f(\ldots f(f(x_1,x_2),x_3),\ldots,x_{m+k+1})$.  Clearly, every $x_i$ is interdefinable with $f_{m+k+1}(x_1,\ldots,x_{m+k+1})$ over the rest of $x_j$ and there are intervals $(c_i,d_i)$  such that  $f_{m+k+1}(x_1,\ldots,x_{m+k+1})\in (b_1,b_2)$  whenever $x_1,x_2\in (c_1,d_1),x_3\in (c_2,d_2)\ldots,x_{m+k+1}\in (c_{m+k},d_{m+k}) $.  Choose  $h_1^{\prime\prime},\ldots,h^{\prime\prime}_{m+k+1}$,  different elements of $\HH^\prime$,  in the appropriate intervals  in such a way that  $a= f_{m+k+1}(h^{\prime\prime}_1,\ldots ,h^{\prime\prime}_{m+k+1})\in (b_1,b_2)$. 
We check that $a$ satisfies the requirements. Let $h_1,\ldots,h_m\in\HH^\prime$. Notice that for some $i$, $h^{\prime\prime}_i\not \in \{h_1,\ldots,h_m,h_1^\prime,\ldots,h_k^\prime\}$, which implies that $h^{\prime\prime}_i$ is not algebraic over $A=\{h_1,\ldots,h_m,h_1^\prime,\ldots,h_k^\prime\}\cup\{h^{\prime\prime}_j\mid j\neq i\}$.  Since   $a$ is interdefinable with $h^{\prime\prime}_i$ over $A$, $a\not \in \acl(h^\prime_1,\ldots,h^\prime_k,h_1,\ldots,h_m)$  and therefore $a\not\in\acl(a_1,\ldots,a_n,h_1,\ldots,h_m)$ and $\c\models\neg\psi(a_1,\ldots,a_n,h_1,\ldots,h_m,a)$.
\end{proof}

\begin{cor}\label{o1}Let $T$ be the theory of an o-minimal expansion of the ordered additive  group of the real numbers. If   $(\c,\HH)$ is the monster model of the theory of $H$-structures of $T$, then  $o(\HH)\neq e(\HH)$.
\end{cor}
\begin{proof} By Proposition~\ref{o0}  applied to  $f(x,y)=x+y$.
\end{proof}

The fact that $(\c,\HH)\equiv (\c,\HH^\prime)$ (where $(\c,\HH)$ is the monster model of $T^\mathrm{indep}$ and $\HH^\prime\supseteq \HH$ is a basis, as in Proposition~\ref{o0}) in the case of  the theory of real-closed fields has been stated, without proof, in Example 2.16 of~\cite{BerensteinVassiliev12}. We thank A. Berenstein for some helpful conversation on these topics.

\begin{prop}\label{o3}  Let $T$ be  an o-minimal theory and assume there are $0$-definable unary functions $f_1,\ldots,f_k$ such that for every set $A$,  $\dcl(A)= f_1(A)\cup\ldots\cup f_k(A)\cup\dcl(\emptyset)$.  If $(\c,\HH)$ is the monster model of the theory $T^\mathrm{indep}$ of $H$-structures of $T$, then $o(\HH)=e(\HH)$.
\end{prop}
\begin{proof} Let $(\c,\HH)\equiv (\c,\HH^\prime)$.  We will use Lemma~\ref{h4}  to prove that $(\c,\HH)\cong (\c,\HH^\prime)$. It is enough to check properties 1 and 2 for the case of a model $B=M\preceq\c$. Assume then $p(x)\in S_1(M)$ is non-algebraic.  By o-minimality (see for instance Theorem 3.3 in~\cite{Pill-Ste86}), the type $p(x)$  determines a cut  $M=B_1\cup B_2$ with $B_1 <B_2$  such that every $a\in \c$ satisfying  $B_1<a<B_2$ realizes $p(x)$.  Since $B_1,B_2$ are small, we can find $a_1,a_2\in \c$  such that  $B_1<a_1<a_2<B_2$. By density of $\HH$  there is some $a\in \HH$ such that $a_1<a<a_2$. But this is expressable in first order and we have that $(\c,\HH) \equiv (\c, \HH^\prime)$,  therefore we also find $a^\prime \in \HH^\prime$ such that $a_1 < a^\prime < a_2$ and hence $a^\prime$ realizes $p(x)$. Now we check property 2.  We use the same notation for the cut determined by $p(x)$  and again we choose $a_1,a_2$ as above. Notice that   $\dcl(M)=M$ and $M\cap (a_1,a_2)=\emptyset$.   By definition of $H$-structure,  in every interval there is some element $a$ such that $a\not\in\dcl(\HH)$. The sentence $$\forall x_1 x_2(x_1<x_2\rightarrow \exists x( x_1<x<x_2 \wedge (\forall y\in H) \bigwedge_{i=1}^k f_i(y)\neq x))$$ holds  in $(\c,\HH)$ and in $(\c,\HH^\prime)$. Hence there is some $a\in (a_1,a_2)$  such that $a\not\in f_1(\HH^\prime)\cup\ldots\cup f_k(\HH^\prime)$. It follows that $a\not\in\dcl(M\HH^\prime)$.  
\end{proof}

\section{Lovely pairs}

Again we consider a geometric theory $T$. Lovely pairs are a generalization of B. Poizat's belle paires studied by I. Ben-Yaacov, A. Pillay and E. Vassiliev (see~\cite{Ben-YaacovPillayVassiliev02}) in the context of simple theories.  A.~Berenstein and E. Vassiliev have adapted in~\cite{BerensteinVassiliev10} the notion of lovely pair to the framework of geometric theories.  This generalizes also L.~van den Dries's  theory of dense pairs of o-minimal expansions of the ordered additive group of real numbers (see~\cite{Dri98})   As in the case of $H$-structures, we recapitulate the basic facts offering short proofs when convenient.  To follow the conventions, we consider now a new unary predicate  $P$  instead of $H$, but the notation $P(A)=A\cap P(M)$, etc.  is similar. Our purpose is to analyse the validity of $e(\PP)=o(\PP)$  when $(\c,\PP)$ is the monster model of the theory of lovely  pairs of $T$.

\begin{defi}\rm  A  \emph{lovely pair}  is a structure  $(M,P(M))$  of language  $L_P=L\cup\{P\}$  such that   $P(M)\preceq  M$  (in $L$)  and
\begin{enumerate}
\item  Coheir: If $A\subseteq M$ is finite  and $q(x)\in S_1(A)$ is a non-algebraic $L$-type, then it has some realization  $a\models q$ in $P(M)$.
\item Extension:  If $A\subseteq M$ is finite  and $q(x)\in S_1(A)$ is a non-algebraic $L$-type, then it has some realization  $a\models q$  in $M$ such that $a\not\in \acl(AP(M))$.
\end{enumerate}
\end{defi}

\begin{prop}\label{l1}  Let $(M,P(M))$ and $(N,P(N))$  be lovely pairs of $T$ and let $I$ be the set of all   $L$-elementary mappings $f$ from $M$ into $N$  with finite domain $A=\dom(f)$   in $M$ and range $B=\rng(f)$ in $N$ and such that  
\begin{enumerate}
\item $a\in P(M)$  iff  $f(a)\in P(N)$  for all $a\in A$
\item $A\ind_{P(A)} P(M)$  and $B\ind_{P(B)}P(N)$.
\end{enumerate}
Then $(M,P(M))$ and $(N,P(N))$  are partially isomorphic via $I$.
\end{prop}
\begin{proof} Similar to the proof of Proposition~\ref{h1}.
\end{proof}

\begin{defi}\rm Let  $T_{P}$ be the theory given by the following axioms:
\begin{enumerate}
\item  $P$  is an elementary substructure:  $(\forall x_1\in P)\ldots (\forall x_n\in P)(\exists x\varphi(x_1,\ldots,x_n,x)\rightarrow (\exists x\in P) \varphi(x_1,\ldots,x_n,x))$\\for every $\varphi(x_1,\ldots,x_n,x)\in L$, for every $n$.
\item Density: $\forall x_1\ldots  x_n ( \exists^\infty x  \varphi(x_1,\ldots,x_n,x)  \rightarrow (\exists x\in P)\varphi(x_1,\ldots,x_n,x))$\\ for every $\varphi(x_1,\ldots,x_n,x)\in L$, for every $n$.
\item Extension: $ \forall x_1\ldots x_n(\forall  y_1\ldots y_m\exists^{<k}x\,\psi(x_1,\ldots,x_n,y_1,\ldots,y_m,x)  \wedge  \exists^\infty x  \varphi(x_1,\ldots,x_n,x) $ \\$ \rightarrow \exists x( \varphi(x_1,\ldots,x_n,x)\wedge (\forall y_1\in P)\ldots (\forall y_m\in P)\neg \psi(x_1,\ldots,x_n,y_1,\ldots,y_m,x))) $
\\ for every $\varphi(x_1,\ldots,x_n,x)\in L$,  $\psi(x_1,\ldots,x_n,y_1,\ldots,y_m,x)\in L$, and  $n,m,k$.
\end{enumerate}
\end{defi}

\begin{prop}\label{l2}  All  lovely pairs are models of $T_{P}$ and  any $\omega$-saturated model of  $T_{P}$  is a  lovely pair.
\end{prop}
\begin{proof} Clear.
\end{proof}

\begin{remark}\label{l3}  If  $(M,H(M))$ is an $H$-structure  and  $P(M)=\acl(H(M))$,  then  $(M,P(M))$  is a lovely pair.
\end{remark}
\begin{proof} See Proposition  4.1 in~\cite{BerensteinVassiliev12}.
\end{proof}

\begin{lemma} \label{l4}  Let $T$ be a geometric theory,  $\c$ the monster model of $T$ and  $(\c,\PP)$  the monster model of the corresponding theory of lovely pairs. 
\begin{enumerate}
\item Density: for every small set $B\subseteq \c$,   every  non-algebraic $L$-type $p(x)\in S_1(B)$  has a realization in $\PP$.
\item Extension: for every small set $B\subseteq \c$,   every  non-algebraic $L$-type $p(x)\in S_1(B)$ has a realization $b$ in $\c$  such that  $b\not\in \acl(B \PP)$.
\item If $(\c^\prime,\PP^\prime)\equiv (\c,\PP)$ has the  properties 1 and 2, then  $(\c,\PP)\cong (\c^\prime,\PP^\prime)$.
\end{enumerate}
\end{lemma}
\begin{proof} Like the proof of Lemma~\ref{h4}.
\end{proof}

\begin{prop} \label{l5} Let $T$ be a geometric theory,  $\c$ the monster model of $T$ and  $(\c,\PP)$  the monster model of the theory  $T_P$ of lovely pairs of models of $T$.   For any small set $A$,  $o(\PP/A)\neq \{\PP\}$.
\end{prop} 
\begin{proof} Like in Proposition~\ref{h5}, now with the help of Lemma~\ref{l4}.
\end{proof}

\begin{prop}\label{l6} Let $T$  have $U$-rank one and let $\c$ be the monster model of $T$. If $(\c,\PP)$  is the monster model of the corresponding theory $T_P$ of lovely pairs, then  $e(\PP)\neq o(\PP)$.
\end{prop}
\begin{proof} Consider the  $H$-structure $(\c,\HH^\prime)$ constructed in the  proof of Theorem~\ref{su1} and  put $\PP^\prime= \acl(\HH^\prime)$. This way we obtain a lovely pair $(\c,\PP^\prime)$ elementarily equivalent to the monster model $(\c,\PP)$ of $T_P$ but not isomorphic to it. The reason is that,  assuming $U$-rank one, we are in a stable theory and, therefore, in the case  1 of the proof of~Theorem~\ref{su1}.   Hence, $\HH^\prime$ is a small set and $\PP^\prime$ is a small model, which clearly implies that $(\c,\PP)\not\cong (\c,\PP^\prime)$.  
\end{proof}

We conjecture that Proposition~\ref{l6} can be generalized to the $SU$-rank one case.  We thank the anonymous referee for some remarks on an earlier version of this result. Case 2 of the proof of~Theorem~\ref{su1} does not seem to give the appropriate result for the corresponding lovely pairs. Of course, if $\acl(A)= A$  for every set $A$, then lovely pairs and $H$-structures coincide and we get the result. But even with a richer algebraic  closure we can prove $e(\PP)\neq o(\PP)$ in some particular cases, as the following example shows.

\begin{ex1} Let $T$ be  the theory  of a vector space $V$ over a finite field $F$ equipped with a non-degenerate simplectic  bilinear form $\beta$ (see~\cite{Hart00}  or~\cite{Granger99} for details). It is an unstable $SU$-rank one theory. Let $\c$ be the monster model of $T$  and let $A_0=\{a_i\mid i<\omega\}$ be a set of orthogonal (i.e., $\beta(a_i,a_j)=0$)  independent elements.  We can carry on a construction similar to that of case 2 in the proof of~Theorem~\ref{su1}, defining sets $A_\alpha$ and $H_\alpha$ for every ordinal $\alpha$ in such a way that for every finite subset $B$ of $A_\alpha$, for every non-algebraic type $r_i(x)\in S_1(B)$  there is some $a_i\in H_{\alpha+1}$ realizing $r_i$ which   is independent of the previous  elements in the $H$-part  and moreover it is orthogonal to all but finitely many elements of $A_0$. That way,  we obtain an $H$-structure $(\c, \HH^\prime)$ with some countable infinite set $A_0\subseteq \HH^\prime$  such that every $a\in \HH^\prime$ is orthogonal to all but finitely many elements of $A_0$.   There is a  type $p(x)\in S_1(A_0)$  containing the formulas  $\beta(x,a)=1$  for every $a\in A_0$  and this type is clearly omitted in $\PP^\prime= \acl(\HH^\prime)$. Hence $(\c,\PP^\prime)$ is a lovely pair not isomorphic to the monster model $(\c,\PP)$ of the theory of lovely pairs of $T$. 
 \end{ex1}

\begin{prop}\label{l7}  Let $T$ be an o-minimal theory such that $\dcl(A)= \bigcup_{a\in A}\dcl(a)$ for every non-empty set $A$.  If $(\c,\PP)$ is the monster model of the theory  $T_P$  of lovely pairs of models of  $T$, then $o(\PP)=e(\PP)$.
\end{prop}
\begin{proof} Similar to the proof of Proposition~\ref{o3}, but starting with $(\c,\PP)\equiv (\c,\PP^\prime)$ and checking that $(\c,\PP^\prime)$ is saturated.  In this case the assumption on $\dcl$ is weaker, but it suffices to check the extension property over arbitrary  elementary submodels $M$ of $\c$: once we have obtained some $a\in \c\smallsetminus\PP^\prime $ in the right cut $B_1 <a<B_2$ using the codensity of $\PP^\prime$, we observe that $a\not\in\dcl(M\PP^\prime)$  since  $a\not \in \PP^\prime =\dcl(\PP^\prime)$ and  $a\not \in M=\dcl(M)$.
\end{proof}

\begin{prop} \label{l8}Let $T$ be the theory of an o-minimal expansion of the ordered additive  group of the real numbers. If   $(\c,\PP)$ is the monster model of the theory of lovely pairs of $T$, then  $o(\PP)\neq e(\PP)$.
\end{prop}
\begin{proof}  $\PP$ is dense and co-dense and we can extend it to some algebraically closed $\PP^\prime\subseteq \c$ with infinite but small algebraic co-dimension, say with co-dimension $\omega$. Notice that $\PP^\prime$ (as well as any algebraically closed set extending $\PP$)  is an elementary submodel of $\c$.   If $\PP^\prime$ contains some interval $(a,b)$, then  by translation for every $x\in \PP^\prime$ there is an interval $(a^\prime,b^\prime)\subseteq \PP^\prime$ such that $x\in(a^\prime,b^\prime)$: take some point $x_0\in (a,b)$ and put  $a^\prime = x-(x_0-a)$ and $b^\prime= x+(b-x_0)$. It follows that  $\PP^\prime$  is a union of intervals and hence it is an open subgroup in the order topology.  Then  $\PP^\prime$ is closed and by density $\PP^\prime= \c$, a contradiction.  This shows that $\PP^\prime$ is co-dense. Obviously, $(\c,\PP)\not \cong (\c,\PP^\prime)$. We now prove that $(\c,\PP^\prime)$ is a lovely pair, which implies $(\c,\PP)\equiv (\c, \PP^\prime)$ and therefore $e(\PP)\neq o(\PP)$. Since $\PP\subseteq \PP^\prime$, every non-algebraic type over a finite set has a realization in $\PP^\prime$. Now we check the extension property.  Consider a finite set $A$ and a non-algebraic type $p(x)\in S_1(A)$ and let us check that $p$ has a realization $a\in \c$  such that $a\not\in\PP^{\prime\prime}=\acl(A,\PP^\prime)$. There is an interval $(b_1,b_2)$ all whose elements satisfy $p$.
Since $A$ has finite dimension, $\PP^{\prime\prime}$ has infinite codimension and exactly as before we see that $\PP^{\prime\prime}$ is co-dense.  Hence, there is some $a\in (b_1,b_2)$ such that $a\not\in\acl(A\cup \PP^\prime)$.
\end{proof}

If $T$ is  the theory of an o-minimal expansion of the ordered additive  group of the real numbers, then $T_P$  can be more easily axiomatized, it is enough to require that $\PP$ is a dense proper elementary substructure (see~\cite{Dri98}). Hence Proposition~\ref{l8} admits a shorter proof, it is enough to consider an extension $\PP^\prime$ of $\PP$ of codimension $1$. We thank A.~Fornasiero for some comments on this.

\noindent{\sc
Departament de Matem\`atiques i Inform\`atica\\
Universitat de Barcelona\\
{\tt e.casanovas@ub.edu}\\

\noindent{\sc
Departamento de Matem\'aticas\\
Universidad de los Andes, Bogot\'a}\\
{\tt lcorredo@uniandes.edu.co}\\

\end{document}